\documentclass[11pt, reqno]{amsart}
\usepackage{cite}
\usepackage{physics}
\usepackage{enumerate}
\usepackage{amssymb}
\usepackage[ansinew]{luainputenc}
\usepackage{amsfonts}
\usepackage{framed}
\usepackage{lipsum}
\usepackage{amsmath,amscd,textcomp}
\usepackage{graphics}
\usepackage{latexsym}
\usepackage{amsthm, amsfonts, graphicx, color} 
\usepackage{mathtools}                 
\usepackage[colorlinks=true, linkcolor=blue, citecolor=blue,urlcolor=blue,bookmarks=false,hypertexnames=true]{hyperref} 
\label{key}\usepackage[dvipsnames]{xcolor}
\usepackage{mathrsfs}
\usepackage[english]{babel}
\usepackage{float}
\usepackage{array}
\usepackage{booktabs} 
\usepackage{orcidlink}
\newcolumntype{S}{>{\small}l}
\allowdisplaybreaks
\newtheorem{thm}{Theorem}[section]

\newtheorem{theorem}[thm]{Theorem}
\newtheorem{lemma}[thm]{Lemma}

\theoremstyle{definition}
\newtheorem{definition}[thm]{Definition}
\theoremstyle{remark}
\newtheorem{remark}[thm]{Remark}
\numberwithin{equation}{section}
\newtheorem{example}[thm]{Example}

\allowdisplaybreaks

\numberwithin{equation}{section}

\DeclareMathOperator*{\supp}{Supp}

\newcommand{\Q}{\sigma}
\newcommand{\y}{\beta}

\newcommand{\T}{\varsigma}

\newcommand{\p}{J}
\newcommand{\J}{j}

\usepackage{hyperref}
\usepackage{color}

\def\supp{{\operatorname{supp}}}

\def\supp{{\operatorname{supp}}}

\usepackage{graphicx}

\DeclareMathSymbol{\subsetneqq}{\mathbin}{AMSb}{36}

\begin{document}
	
	\begin{flushleft}
		{\bf\Large {Special-Affine Wavelets : Multi Resolution Analysis and Function Approximation in $\boldsymbol{L^2(\mathbb R)}$}}
	\end{flushleft}
	
	\parindent=0mm \vspace{.4in}
	
	{\bf{Vikash K. Sahu$^{a,\orcidlink{0009-0005-7387-7836}}$, Waseem Z. Lone$^{a, \orcidlink{0000-0002-9826-9475}}$, and Amit K. Verma$^{a,\orcidlink{0000-0001-8768-094X}}$ }}
	
	\parindent=0mm \vspace{.1in}
	{\small \it $^{a}$Department of Mathematics, Indian Institute of Technology Patna, Bihta, Patna 801103, (BR) India. \\[2mm]
		E-mail: $\text{vikash\_2421ma11@iitp.ac.in}$;\,$\text{waseem\_24ps41@iitp.ac.in};\,$\text{akverma@iitp.ac.in}} \\
	
	$^{\star}$Corresponding author: $\text{akverma@iitp.ac.in}$ 
	
	\parindent=0mm \vspace{.2in}
	{\small {\bf Abstract.} The multiresolution analysis (MRA) associated with the Special affine Fourier transform (SAFT) provides a structured approach for generating orthonormal bases in \( L^2(\mathbb R) \), making it a powerful tool for advanced signal analysis. This work introduces a robust sampling theory and constructs multiresolution structures within the SAFT domain to support the formation of orthonormal bases. Motivated by the need for a sampling theorem applicable to band-limited signals in the SAFT framework, we establish a corresponding theoretical foundation. Furthermore, a method for constructing orthogonal bases in $L^2(\mathbb R)$ is proposed, and the theoretical results are demonstrated through illustrative examples.
		
		\parindent=0mm \vspace{.1in}
		{\bf{Keywords:}}  Fourier transform; Special affine Fourier transform; Shannon's sampling theorem; Multiresolution analysis; Special affine wavelet transform; Orthonormal basis.
		
		\parindent=0mm \vspace{.1in}
		{\bf {Mathematics Subject Classification:}} 42C40; 46E30; 42A38; 44A05; 94A12.}
	
	\def\subjclassname{{\rm 2025} {\it Mathematics Subject Classification} }

	\section{Introduction}
	The Fourier transform (FT) is a fundamental mathematical tool in mathematical analysis and and signal processing that decomposes signals into its constituent frequency components. Since its inception, the FT has found widespread applications across numerous domains in science and engineering \cite{Kith, Debnath1,Debnath2}. For  any $f\in L^2(\mathbb R)$, the FT of is defined by
	\begin{align}\label{S1E1}
		\mathcal{F}\big[f\big](\zeta) =	\hat{f}(\zeta)=\dfrac{1}{\sqrt{2\pi}} \int_\mathbb R f(x)\, e^{-i\zeta x}\,dx.
	\end{align}	
	
	\parindent=8mm\vspace{.05in}	
	Over time, various generalizations of the FT have emerged to address limitations in handling non-stationary and chirp signals. One such notable generalization is the fractional Fourier transform (FrFT),  which performs a rotation in the time-frequency plane and is well-suited for analyzing chirp-like signals \cite{OZ,Almeida,Namias}. The FrFT of order $\alpha$ of a signal $f(t)$ is defined as
	\begin{align}\label{S1E2}
		\mathcal{F}_\alpha\big[f\big](\zeta) = \sqrt{\dfrac{1 - i \cot \alpha}{2\pi}} \int_{\mathbb R} f(t) \exp\Big\{ i\pi\left( x^2 + \zeta^2 \right)\cot\alpha - 2\pi i\, x \zeta\, \csc \alpha \Big\}\, dx,~~\alpha \notin \{ k\pi : k \in \mathbb{Z} \}.
			\end{align}	
	
	\parindent=8mm\vspace{.05in}
	A significant extension to the theory of FTs was introduced by Moshinsky and Quesne \cite{Moshinsky}, in the form of linear canonical transform (LCT), which unifies several well-known integral transforms, including the FrFT. The LCT of a signal $f(t)$ with respect to a real $2 \times 2$ matrix $M=\left[\begin{smallmatrix} A & B \\ C & D \end{smallmatrix}\right]$ satisfying $AD - BC = 1$ is defined by
	\begin{align}\label{S1E3}
		\mathcal{L}_{M}\big[f\big](\zeta) = \dfrac{1}{\sqrt{2\pi i B}} \int_{\mathbb R} f(x) \exp\Big\{ \tfrac{i}{2B} \left(A x^2 - 2 x \zeta + D \zeta^2 \right) \Big\} \, dx, \quad  B \neq 0.
	\end{align}
	
	\parindent=8mm\vspace{.05in}
	Despite these advancements, there remained a need for a unified and flexible framework that could encompass these transforms under a single structure. This led to the development of the special affine Fourier transform (SAFT). The SAFT introduces additional flexibility by incorporating time shifts and frequency modulations in the LCT kernel \cite{Healy,Abe1,Abe2}. Given a real uni-modular augmented matrix $S=\left[\begin{smallmatrix} A & B & : & p \\ C & D & : & q \end{smallmatrix}\right]$ with $B\neq 0$, the SAFT of a function $f \in L^2(\mathbb R)$ is given by 
	\begin{align*}
		\mathscr{O}_S\big[f\big](\zeta) =  \frac{1}{\sqrt{2\pi iB}}\int_{\mathbb R} f(x) \exp\Big\{\tfrac{i}{2B} \left( A x^2 + 2x(p - \zeta) - 2\zeta(Dp - Bq) + D\left(\zeta^2 + p^2\right)\right)\Big\}\,dx.
	\end{align*}
	Owing to its additional degrees of freedom, the SAFT has achieved significant success across various domains, including optics, signal processing, electrical and communication systems, quantum mechanics, and numerous other areas of science and engineering \cite{Kumar1,Kumar2,Hitzer,Lone, Teali,Shah,Bhandari,Xiang,Qiang}.
	
	\parindent=8mm\vspace{.1in}
	Besides a lot of advantages, the SAFT has one major drawback due to its global kernel; that is, it only provides spectral information with no indication about the time localization of the SAFT spectral components of the given signal. This limitation is completely addressed by the special affine wavelet transform (SAWT) \cite{Shah1}. The essence of the SAWT is that it employs reasonably flexible window functions endowed with higher degrees of freedom called as special affine wavelets.
	
	\parindent=8mm\vspace{.1in}	
	One of the prime advantages of wavelets is that they can provide wavelet bases in which the basis functions are constructed by dilating and translating the mother wavelet. In the late 1986, Meyer and Mallat recognized that the construction of different wavelet bases can be realized by the so-called multiresolution analysis (MRA) \cite{Meyer,Hernhdez}. This general formalism of MRA has gained considerable attention and has offered an alternative approach for the construction of wavelets. 
	
	\parindent=8mm\vspace{.1in}
	Shi et al. firstly proposed a fractional version of MRA in \cite{15}, while Ahmad \cite{16} further investigated fractional MRA and related scaling functions in the $L^2(\mathbb R)$ space. Building on these ideas, Dai et al. \cite{17} introduced the fractional wavelet transform (FRWT), developed an associated MRA structure, and constructed corresponding orthogonal fractional wavelets. Shah and Lone \cite{18,19} explored the notion of special affine MRA and established orthonormal wavelets in $L^2(\mathbb R)$ by discretizing the continuous special affine wavelets.  More recently, Verma et al. \cite{Verma} constructed the MRA associated with the quadratic-phase Fourier transform.
	
	\parindent=8mm\vspace{.1in}
	Keeping in view the elegance of the MRA and the structural flexibility of the SAFT, we propose a novel framework called special affine multiresolution analysis (SAMRA). This framework is built by defining a family of nested subspaces, introducing explicit scaling functions, and constructing corresponding orthonormal bases using modulated sinc functions. Unlike the approach developed by Shah and Lone \cite{18,19}, which is based on discretizing continuous special affine wavelets, our formulation leverages the sampling theorem in the SAFT domain as the foundation for constructing SAMRA. This comprehensive structure not only generalizes the classical MRA but also provides greater adaptability and analytical depth. 
	
	\parindent=8mm\vspace{.1in}
	The remainder of the paper is organized as follows: Section \ref{S2} presents the preliminary analysis and explores the basic sampling theorem in the QPFT domain. In Section \ref{S3}, we construct orthonormal special affine wavelets within the framework of SAMRA. Finally, Section \ref{S4} provides concluding remarks.
	
	\section{Preliminaries}\label{S2}
	
	\parindent=0mm\vspace{.0in}
	In this section, we present a concise overview of the SAFT. For notational clarity, we represent a $2 \times 2$ real augmented matrix $S=\left[\begin{smallmatrix} A & B & : & p \\ C & D & : & q \end{smallmatrix}\right]$ in the compact form $S = [M \mid \Lambda]$, where $M=\left[\begin{smallmatrix} A & B \\ C & D  \end{smallmatrix}\right]$ is the core transformation matrix and $\Lambda = (p, q)^\top$ encodes the affine parameters. Throughout this work, we restrict our attention to the case $\det M = 1$ and $B > 0$, as these conditions ensure the invertibility of the associated transform and preserve the essential properties of the SAFT.
	
	\begin{definition}\label{S2D1}
		Given a parametric matrix $S = [M \mid\Lambda]$, the SAFT of a function $f \in L^2(\mathbb R)$ is denoted by $\mathscr{O}_{S}[f](\zeta)$ and is defined as			
		\begin{align}\label{S2E1}
			\mathscr{O}_{S}\big[f\big](\zeta) = \int_{\mathbb R} f(x) \, \mathcal K_S (x,\zeta) \, dx,
		\end{align}
		where $\mathcal K_S(x,\zeta)$  is the SAFT kernel associated with the matrix $S$ and is given by		
		\begin{align}\label{S2E2}
			\mathcal K_S(x,\zeta) = 
			\begin{cases}
				\frac{1}{\sqrt{2\pi i B} }\exp \Big\{\tfrac{i}{2B} \left( A x^2 + 2x(p - \zeta) - 2\zeta(Dp - Bq) + D\left(\zeta^2 + p^2\right)\right)\Big\}, & B\neq 0
				\\[3mm]
				\sqrt{D}  \exp \Big\{\frac{i C D (\zeta - p)^2}{2} + i q \zeta \Big\} f\big(D(\zeta - p)\big) , & B=0.
			\end{cases}
		\end{align}			
	\end{definition}	
	The inverse transform corresponding to \eqref{S2E2} is given by
	\begin{align}\label{S2E3}
		f(x) = \mathfrak{I}_{S} \int_{\mathbb R} \mathscr{O}_S \big[f\big](\zeta) \,\mathcal K_{S^{-1}}(\zeta, x)\, d\zeta,
	\end{align}
	where the prefactor $\mathfrak{I}_{S}$ is defined as
	\begin{align*}
		\mathfrak{I}_{S} = \exp \Big\{ \tfrac{i}{2} \left( CDp^2 + ABq^2 - 2ADpq \right) \Big\},
	\end{align*}
	and $S^{-1} = \left[M^{-1}\mid \Lambda^{-1}\right]$ with 
	$$M^{-1}=\left[\begin{matrix} 	D & -B \\ -C & A \end{matrix}\right]\quad \mbox{and} \quad \Lambda^{-1}=\left[\begin{matrix} 	Dp - Bq \\ Aq - Cp\end{matrix}\right].$$
	
	\begin{remark}\label{S2R2}
		The inverse kernel $\mathcal{K}_{S^{-1}}(x,\zeta)$ is related to the forward kernel via the identity $\mathcal K_{S^{-1}}(x,\zeta) = \overline{\mathfrak{I}_{S}\, \mathcal K(x,\zeta)}$.
	\end{remark}
	
	Moreover, the Parseval's formula for the special affine Fourier transform is given by
	\begin{align}\label{1}
		\big\langle f,\, g\big\rangle=\Big\langle \mathscr O_{S}\big[f\big],\, \mathscr O_{S}\big[g\big]\Big\rangle,\quad \forall~ f,g\in L^2(\mathbb R).
	\end{align}

	The following lemma plays a crucial role in the formulation of the SAMRA.	
	\begin{lemma}\label{S2L3}
		Let $f \in L^2(\mathbb R)$ and $ a\neq 0$. Then the following identity holds:
		\begin{align}\label{S2E4}
			&\rm{(i).}\quad\mathscr{O}_{S} \big[f(ax)\big] (\zeta) = \frac{1}{|a|} \exp\Big\{ \tfrac{i}{2B} \left( 2\zeta(Dp - Bq) \left( \tfrac{1}{a} - 1 \right) - \zeta^2 D \left( \tfrac{1}{a^2} - 1 \right) \right) \Big\}\notag 
			\\
			&\phantom{-------.----} \times  \mathscr{O}_{S} \Big[ \exp\Big\{ \tfrac{i}{2B} \left( Ax^2 \left( \tfrac{1}{a^2} - 1 \right) + 2Dx \left( \tfrac{1}{a} - 1 \right) \right) \Big\} f(x) \Big]  \left( \tfrac{\zeta}{a} \right), 
			\\
			&\rm{(ii).}\quad\mathscr{O}_{S} \left[ \exp\Big\{-\tfrac{i}{2B}\left(Ax^2\big(1-a^2\big)+2xp(1-a)\right)\Big\}f(ax) \right] (\zeta) \notag 
			\\
			&\label{S2E5} \phantom{-----.---}=\frac{1}{|a|} \exp\Big\{ \tfrac{i}{2B} \left( 2\zeta(Dp - Bq) \left( \tfrac{1}{a} - 1 \right) - \zeta^2 D \left( \tfrac{1}{a^2} - 1 \right) \right) \Big\} \, \mathscr{O}_{S} \big[f\big ]\left(\tfrac{\zeta}{a}\right).
		\end{align}		
	\end{lemma} 
	
	\begin{proof}
		(i). Using the definition of the SAFT, we have
		\begin{align*}
			&\mathscr{O}_{S}\left[f(ax)\right](\zeta) = \frac{1}{\sqrt{2\pi i B} } \int_{\mathbb R} f(ax) \exp \Big\{\tfrac{i}{2B} \left( A x^2 + 2x(p - \zeta) - 2\zeta(Dp - Bq) + D\left(\zeta^2 + p^2\right)\right)\Big\} \,dx
			\\
			& = \frac{1}{|a| \sqrt{2iB\pi}} \int_{\mathbb R} f(u)\exp\left\{\tfrac{i}{2B}  \left(\tfrac{Au^2}{a^2} + \tfrac{u}{a}\left(p - \zeta\right) - 2\zeta(Dp - Bq) + D\left(\zeta^2 + p^2\right)\right)\right\} du
			\\
			& = \frac{1}{|a| \sqrt{2iB\pi}} \exp\Big\{\tfrac{i}{2B} \left(- 2\zeta\left(Dp - Bq\right) + D\left(\zeta^2 + p^2\right)\right)\Big\} \int_{\mathbb R} \exp\left\{\tfrac{i}{2B}\left( {\tfrac{Au^2}{a^2} + \tfrac{2u}{a}(p - \zeta) }\right)\right\} f(u) \, du
			\\
			& = \frac{1}{|a| \sqrt{2iB\pi}} \exp\Big\{\tfrac{i}{2B} {\left(- 2\zeta(Dp - Bq) + D\left(\zeta^2 + p^2\right)\right)}\Big\}  \int_{\mathbb R}  \exp\left\{\tfrac{i}{2B}     {\left(\tfrac{Au^2}{a^2} + \tfrac{2up}{a} \right) }\right\}f(u) 
			\\ 
			&\qquad\times  \exp\left\{-\tfrac{i2u\zeta}{2Ba}\right\} \exp\left\{\tfrac{i}{2B} {\left(Au^2+2up-\tfrac{2\zeta}{a}(Dp - Bq) + D\left(\tfrac{\zeta^2}{a^2} + p^2\right)\right)}\right\}
			\\
			&\qquad\times\exp\left\{-\tfrac{i}{2B} {\left(Au^2+2up- \tfrac{2\zeta}{a}(Dp - Bq) + D\left(\tfrac{\zeta^2}{a^2} + p^2\right)\right)}\right\}du
			\\
			&=\frac{1}{|a| \sqrt{2iB\pi}}  \exp\Big\{ \tfrac{i}{2B} \left( 2\zeta(Dp - Bq) \left( \tfrac{1}{a} - 1 \right) - D\zeta^2\left(1+\tfrac{1}{a^2}\right)  \right) \Big\}
			\\
			&\qquad\times\int_{\mathbb R}f(u)\exp\left\{\tfrac{i}{2B}\left(Au^2+2u(p-\tfrac{\zeta}{a})-\tfrac{2\zeta}{a}(dp-Bq)+D\left(\tfrac{\zeta^2}{a^2}+p^2\right)\right)\right\}
			\\
			&\qquad\times \exp\Big\{ \tfrac{i}{2B} \left( Au^2 \left( \tfrac{1}{a^2} - 1 \right) + 2up \left( \tfrac{1}{a} - 1 \right) \right) \Big\} \, du
			\\
			& = \frac{1}{|a|} \exp\Big\{ \tfrac{i}{2B} \left( 2\zeta(Dp - Bq) \left( \tfrac{1}{a} - 1 \right) - D\zeta^2\left(1+\tfrac{1}{a^2}\right)  \right) \Big\}
			\\
			&\quad\times\int_{\mathbb R} \exp\Big\{\tfrac{i}{2B} \left( Au^2 \left( \tfrac{1}{a^2} - 1 \right) + 2up \left( \tfrac{1}{a} - 1 \right) \right)  f(u)\Big\} \mathcal K_S\left(u,\tfrac{\zeta}{a}\right) du
			\\
			&=\frac{1}{|a|} \exp\Big\{ \tfrac{i}{2B} \left( 2\zeta(Dp - Bq) \left( \tfrac{1}{a} - 1 \right) - D\zeta^2\left(1+\tfrac{1}{a^2}\right) \right) \Big\}
			\\
			&\quad \times  \mathscr{O}_{S} \left[  \exp\Big\{ \tfrac{i}{2B} \left(Au^2 \left( \tfrac{1}{a^2} - 1 \right) + 2up \left( \tfrac{1}{a} - 1 \right) \right) f(u) \Big\}\right]  \left( \tfrac{\zeta}{a} \right).
		\end{align*}

		(ii). We now proceed as:
		\begin{align*}
			\mathscr{O}_{S}\left[h\big(x\big)f\big(x\big)\right] (\zeta) &=\int_{\mathbb R}h(x)f(ax) \, \mathcal K_S(x,\zeta) \, dx\\
			&=\frac{1}{|a|} \int_{\mathbb R}h(\tfrac{u}{a})f(u) \, \mathcal K_S\left(\tfrac{u}{a},\zeta\right)du\\
			&= \frac{1}{|a|} \exp\Big\{\tfrac{i}{2B} \left( 2\zeta(Dp - Bq) \left( \tfrac{1}{a} - 1 \right) - \zeta^2 D \left( \tfrac{1}{a^2} - 1 \right) \right) \Big\}\\
			&\quad\times \mathscr{O}_{S} \Big[ \exp\left\{\tfrac{i}{2B}\left(Au^2(\tfrac{1}{a^2}-1)+2up(\tfrac{1}{a}-1)\right) \right\} h\left(\tfrac{u}{a}\right)f(u) \Big]\left(\tfrac{\zeta}{a}\right).
		\end{align*}
		Setting $h\left(\tfrac{u}{a}\right)=\exp\left\{-\tfrac{i}{2B}\left(Au^2(\tfrac{1}{a^2}-1)+2up(\tfrac{1}{a}-1)\right)\right\}$ yields
		\begin{align*}
			&\mathscr{O}_{S}\left[\exp\big\{-\tfrac{i}{2B}\left(Au^2(\tfrac{1}{a^2}-1)+2up(\tfrac{1}{a}-1)\right)\Big\}\right]\\
			&\qquad\qquad= \frac{1}{|a|} \exp\left\{\tfrac{i}{2B} \left(2\zeta(Dp - Bq)\left( \tfrac{1}{a} - 1\right) - \zeta^2 D\left( \tfrac{1}{a^2} - 1 \right)\right) \right\} \mathscr{O}_{S}\big[f\big]\left(\tfrac{\zeta}{a}\right).
		\end{align*}
		This completes the proof.
	\end{proof}
	
	The following lemma is a fundamental component in the formulation of Shannon's sampling theorem in the SAFT domain.
	
	\begin{lemma}\label{S2L4}
		Assume that a signal $f(x)$ is band-limited to the interval $(-\Omega_S,\Omega_S)$ in the special affine Fourier domain associated with the parameter matrix $S = [M \mid \Lambda]$. Define
		\begin{align} \label{S2E6}
			g(x)= \int_{\mathbb R}  \mathscr{O}_{S}\big[f\big](\zeta)
			\exp \left\{ -\tfrac{i}{2B} \left( -2x\zeta-2\zeta(Dp-Bq)+D\zeta^2\right)\right\} \, d\zeta.
		\end{align}			 
		Then, the function $g(x)$ is band-limited to the interval $\left( \tfrac{-\Omega_S}{B}, \tfrac{\Omega_S}{B} \right)$ in the classical Fourier domain.
	\end{lemma}
	
	\begin{proof}
		We have,
		\begin{align*}
			g(x)=\int_{\mathbb R}  \mathscr{O}_{S}\big[f\big](\zeta)
			\exp\left\{-\tfrac{i}{2B} \left( -2x\zeta-2\zeta(Dp-Bq)+D\zeta^2\right)\right\}d\zeta.
		\end{align*}
		Taking the Fourier transform on both sides gives
		\begin{align*}
			\mathcal{F}\big[g\big] (\omega) &= \int_{\mathbb R} \left(\int_{\mathbb R}  \mathscr{O}_{S}\big[f\big](\zeta)
			\exp\left\{ -\tfrac{i}{2B} \left(-2\zeta(Dp-Bq)+D\zeta^2\right)\right\}d\zeta	\right) e^{-ix\omega} dx
			\\
			&= \int_{\mathbb R} \mathscr{O}_{S} \big[f\big](\zeta)
			\exp \left\{-\tfrac{i}{2B} \left(-2\zeta(Dp-Bq)+D\zeta^2\right)\right\} \left(\int_{\mathbb R} \exp\left\{-ix\left(\tfrac{-\zeta}{B}+\omega\right)\right\}dx\right) d\zeta
			\\
			&=\sqrt{2\pi}\displaystyle \int_{\mathbb R}
			\mathscr{O}_{S}\big[f\big]\left(\zeta\right)
			\exp\left\{-\tfrac{i}{2B} \left( -2\zeta(Dp-Bq)+D\zeta^2\right)\right\}\delta\left(\tfrac{-\zeta}{B}+\omega\right) d\zeta
			\\
			&=-B\sqrt{2\pi} \int_{\mathbb R} 
			\mathscr{O}_{S}\big[f\big](-B\zeta)
			\exp\left\{-\tfrac{i}{2B} \left( -2B\zeta(Dp-Bq)+DB^2\zeta^2\right)\right\}\delta({\zeta}-\omega)d\zeta\\
			&=-\sqrt{2\pi}B ~\mathscr{O}_{S}\big[f\big]\left(-B\omega\right)
			\exp\left\{-\tfrac{i}{2B} \left(-2B\omega(Dp-Bq)+DB^2\omega^2\right)\right\}.
		\end{align*}
		Since	$\mathscr{O}_{S}\big[f\big](\omega)$ is a  band-limited to   $(-\Omega_{S},  \Omega_{S})$, so supp \,$\mathscr{O}_{S}\big[f\big](\omega)\subset\ (-\Omega_{S},  \Omega_{S})$, i.e.,
		\begin{align*}
			& \mathscr{O}_{S} \big[f\big](\omega) \neq 0 \quad \text{a.e.},\quad |\omega| > \Omega_S, \\
			& \omega' =- B\omega \quad \Rightarrow \quad \mathscr{O}_{S}\big[f\big](\omega') \neq 0 \quad \text{a.e.},\quad |\omega'| > \Omega_S, \\
			& |\omega'| = |-B\omega| > \Omega_S \quad \Rightarrow \quad |\omega| > \tfrac{\Omega_S}{B}, \\
			& \therefore \mathscr{O}_{S}\big[f\big](B\omega) \neq 0 \quad \text{a.e.},\quad |\omega| > \tfrac{\Omega_S}{B}.
		\end{align*}
		Thus  $\supp~\mathcal {F}\big[g\big] (\omega)\neq 0$  a.e., $|\omega| >\tfrac{\Omega_{S}}{B}$.
		
		\parindent=0mm\vspace{.1in}
		Hence
		$g(x)$ is band -limited to $\left| \omega \right| >\tfrac{\Omega_{S}}{B}$ in the Fourier domain.
	\end{proof}
	
	The Shannon sampling theorem is one of the most profound and foundational concepts in digital signal processing, serving as a bridge between analog and digital signals. It provides the essential condition under which a continuous-time analog signal can be perfectly reconstructed from its discrete samples. Mathematically, 
	\begin{align}\label{S2E7}
		f(x)=\sum_{n\in\mathbb Z} f\left(\tfrac{n\pi}{\zeta}\right)\mathrm{sinc}\left(\tfrac{x\zeta -n\pi}{\pi}\right).
	\end{align}The sampling theorem in the context of the SAFT is particularly important, as many chirp-like signals are not band-limited in the conventional Fourier domain but are band-limited in generalized Fourier domains. The sampling theorem for SAFT band-limited signals is presented below.
	
	\begin{theorem}\label{S2T8}
		Let the signal $f(x)$ be band-limited to $(-\Omega_S,\Omega_S)$  in the SAFT domain associated with the parameter matrix $S = [M \mid \Lambda]$. Then, the following sampling expansion holds for $f(x)$:
		\begin{align}\label{S2E8}
			f(x)= \exp\left\{-\tfrac{i}{2B}\left(Ax^2+2xp\right)\right\}\sum_{n \in \mathbb{Z}}f(nT) \exp\left\{\tfrac{i}{2B}{\left(A(nT)^2 + pnT\right)}\right\} \mathrm{sinc} \left(\tfrac{\Omega_S{(x-nT)}}{B\pi}\right).
		\end{align}
	\end{theorem}
	
	\begin{proof}
		We have 
		\begin{align*}
			f(x)&=\frac{1}{\sqrt{2i\pi B}}\int_{\mathbb R}\mathscr{O}_{S}\big[f\big](\zeta) \,
			\overline{\mathfrak{I}_{S} \mathcal K(x,\zeta)} \, d\zeta
			\\
			&=\frac{\bar{\mathfrak{I}}_{S}}{\sqrt{2i\pi B}}\int_{\mathbb R} \mathscr{O}_{S} \big[f\big](\zeta)	  
			\exp \left\{ -\tfrac{i}{2B}\left(Ax^2+2xp+Dp^2\right) \right\}	
			\\ &\qquad\qquad\qquad\times\exp\left\{-\tfrac{i}{2B}\left(-2x\zeta-2\zeta\left(Dp-Bq+D\zeta^2\right)\right)\right\}d\zeta
			\\
			&=\frac{\bar{\mathfrak{I}}_{S}}{\sqrt{2i\pi B}}  \exp\left\{-\tfrac{i}{2B}\left(Ax^2+2xp+Dp^2\right)\right\}
			\\ 
			&\qquad\qquad\qquad\times \int_{\mathbb R}\mathscr{O}_{S}\big[f\big](\zeta)
			\exp\left\{-\tfrac{i}{2B}\left(-2x\zeta-2\zeta\left(Dp-Bq+D\zeta^2\right)\right)\right\}d\zeta\\
			&=\frac{\bar{\mathfrak{I}}_{S}}{\sqrt{2i\pi B}}  \exp\left\{-\tfrac{i}{2B}\left(Ax^2+2xp+Dp^2\right)\right\}g(x),
		\end{align*}
		where
		$$g(x)=\int_{\mathbb R}\mathscr{O}_{S}\big[f\big](\zeta) 
		\exp \left\{ -\tfrac{i}{2B}\left(-2x\zeta-2\zeta \left(Dp-Bq+D\zeta^2\right)\right)\right\} d\zeta.$$
		Since $g(x)$ is band-limited to the interval $\left(-\Omega_S/B, \Omega_S/B\right)$ in the Fourier domain, the classical Shannon sampling theorem \eqref{S2E8} yields:
		\begin{align*}
			g(x)=\sum_{n \in \mathbb{Z}}g(nT)\,\text{sinc}\Big(\tfrac{\Omega_S{(x-nT)}}{B\pi}\Big),
		\end{align*}
		where $T =\tfrac{B\pi}{\Omega_S}$ is a sample period. Therefore,
		\begin{align*}
			f(x)&=\frac{\bar{\mathfrak{I}}_{S}}{\sqrt{2i\pi B}}\exp\left\{-\tfrac{i}{2B}\left(Ax^2+2xp+Dp^2\right)\right\}\sum_{n \in \mathbb{Z}}g(nT) \,\text{sinc}\Big(\tfrac{\Omega_S{(x-nT)}}{B\pi}\Big)
			\\
			&=\frac{\bar{\mathfrak{I}}_{S}}{\sqrt{2i\pi B}}\exp\left\{-\tfrac{i}{2B}\left(Ax^2+2xp+Dp^2\right)\right\}
			\\
			&\hspace{0.5cm}\times \sum_{n \in \mathbb{Z}}{\sqrt{2i\pi B}}\,{\mathfrak{I}_{S}}\exp\left\{\tfrac{i}{2B}\left(A(nT)^2+2nTp+Dp^2\right)\right\}f(nT) \,\text{sinc}\Big(\tfrac{\Omega_S{(x-nT)}}{B\pi}\Big)
			\\
			&=\exp\left\{-\tfrac{i}{2B}\left(Ax^2+2xp+Dp^2\right)\right\}\sum_{n \in \mathbb{Z}} \exp\left\{\tfrac{i}{2B}\left(A(nT)^2+2nTp+Dp^2\right)\right\} f(nT)
			\\
			&\qquad\times \text{sinc}\Big(\tfrac{\Omega_S{(x-nT)}}{B\pi}\Big).   
		\end{align*}
		The proof is now complete.	
	\end{proof}

	\section{Multiresolution Analysis Associated with the SAFT}\label{S3}
	
	\parindent=0mm\vspace{.0in}
	In this section, we develop a MRA framework associated with the SAFT. This construction is inspired by the fundamental ideas underlying Shannon's sampling theorem, as discussed previously. To formally define the SAFT-based MRA (SAMRA), we begin with the following preliminary discussion. 
	
	\parindent=8mm\vspace{.1in}
	When the bandwidth in the SAFT domain is given by $\Omega_{S}=B\pi$, , the corresponding space of band-limited signals is denoted by $\mathscr{V}_{S}^0$, defined as
	\begin{align*}
		\mathscr{V}_{S}^0=\Big\{f \in  L^2(\mathbb R): \mathscr{O}_{S}\big[f\big](\zeta) = 0, ~\text{for}~ |\zeta|\geq\Omega_S, \, \Omega_S=B\pi\Big\},
	\end{align*}
	where sampling period $T = 1$. Therefore, according to Theorem \ref{S2T8}, any function $ f\in \mathscr{V}_{S}^0$ can be reconstructed as
	\begin{align*}
		f(x)=\sum_{n \in \mathbb{Z}}\exp\left\{-\tfrac{i}{2B}\left(Ax^2+2xp+Dp^2\right)\right\} \exp\left\{\tfrac{i}{2B}\left(An^2+2np+Dp^2\right)\right\} f(n)\,\text{sinc}\Big(\tfrac{\Omega_S{(x-n)}}{B\pi}\Big).
	\end{align*}
	Given that $\Omega_S = B\pi$, it follows that
	\begin{align*}
		f(x)&=\sum_{n \in \mathbb{Z}}\exp\left\{-\tfrac{i}{2B}\left(A(x^2-n^2)+2p(x-n)\right)\right\}  f(n) \, \text{sinc}{\left(x-n\right)}
		\\
		&=\sum_{n \in \mathbb{Z}}f(n)\,\varphi_{S,0,n}(x),
	\end{align*}
	where 
	\begin{align}\label{S3E1}
		\varphi_{S,0,n}(x) =\exp\left\{-\tfrac{i}{2B}\left(A(x^2-n^2)+2p(x-n)\right)\right\}  \text{sinc}{\left(x-n\right)}.
	\end{align}
	
	\parindent=8mm\vspace{.1in}
	Utilizing the orthogonality of the set $\left\{\varphi_{S,0,n}:n \in \mathbb Z\right\}$, we further obtain that the collection $\big\{\varphi_{S,0,n}(x):n\in\mathbb Z\big\}$ constitutes an orthonormal basis for the space $\mathscr{V}_{S}^0$.
	
	\parindent=8mm\vspace{.1in}
	When $\Omega_{S} = 2B\pi$ and the sampling period is $T = 1/2$, the corresponding set of band-limited signals in the SAFT domain is denoted by $\mathscr{V}_{S}^1$, defined as
	$$\mathscr{V}_{S}^1=\Big\{f \in  L^2(\mathbb R): \mathscr{O}_{S}\big[f\big](\zeta) = 0, \,|\zeta |\geq \Omega_S=2B\pi\Big\}.$$
	According to the Theorem \ref{S2T8}, the following chain of inclusions holds: 
	\begin{align*}
		g \in \mathscr{V}_{S}^0 &\implies\mathscr{O}_{S}\big[f\big](\zeta) = 0, ~ | \zeta |\geq B\pi\\
		&\implies\mathscr{O}_{S}\big[f\big](\zeta) = 0,  ~ | \zeta |\geq \ 2B\pi\\
		&\implies g\in \mathscr{V}_{S}^1,
	\end{align*}
	which shows that $\mathscr{V}_{S}^0 \subset  \mathscr{V}_{S}^1$.
	
	\parindent=8mm\vspace{.1in}	
	Moreover, when  $T=1/2$ and $\Omega_S= 2\pi B$, any function $f \in \mathscr{V}_{S}^1$ can be represented as
	\begin{align*}
		f(x)&=\sum_{n \in \mathbb{Z}}\exp\left\{-\tfrac{i}{2B}\left(Ax^2+2xp\right)\right\} \exp\left\{\tfrac{i}{2B}\left(\tfrac{An^2}{4}+np\right)\right\} f\left(\tfrac{n}{2}\right)\text{sinc}\left({x-\tfrac{n}{2}}\right)
		\\
		&=\sum_{n \in \mathbb{Z}}\exp\left\{-\tfrac{i}{2B}\left(A(x^2-\tfrac{n^2}{4})+ 2p\left(x-\tfrac{n}{2}\right)\right)\right\} f\left(\tfrac{n}{2}\right) \text{sinc}{\left(2x-n\right)}
		\\
		&=\frac{1}{\sqrt{2}}\sum_{n \in \mathbb{Z}}f\left(\tfrac{n}{2}\right)\varphi_{S,1,n}(x),
	\end{align*}
	where 
	\begin{align}\label{S3E2}
		\varphi_{S,1,n}(x)=\sqrt{2}\exp\left\{-\tfrac{i}{2B}\left(A\left(x^2-\tfrac{n^2}{4}\right)+  2p\left(x-\tfrac{n}{2}\right)\right)\right\}  \text{sinc}{\left(2x-n\right)}.
	\end{align}
	
	\parindent=8mm\vspace{.1in}
	It follows further that the set $\left\{\varphi_{S,1,n} : n \in \mathbb{Z}\right\}$ constitutes an orthonormal basis for the space $\mathscr{V}_{S}^1$. In particular, for $a = 2$, relation \eqref{S2E5} takes the form 
	\begin{align*}
		&\mathscr{O}_{S} \left[ \exp\left\{-\tfrac{i}{2B}\left(-3Ax^2-2xp\right)\right\} f(2x) \right] (\zeta) =\tfrac{1}{2} \exp\left\{ \tfrac{i}{2B} \left(-\zeta(Dp - Bq)+  \tfrac{3}{4}\zeta^2 D \right) \right\}\mathscr{O}_{S} \big[f\big]\left(\tfrac{\zeta}{2}\right).
	\end{align*}
	Therefore, $f \in \mathscr{V}_{S}^0$  if and only if  $\exp\left\{\tfrac{i}{2B}\left(3Ax^2+2xp\right)\right\}f(2x)$ belongs to $\mathscr{V}_{S}^1 $. This equivalence arises from the fact that
	\begin{align*}
		f \in \mathscr{V}_{S}^0&\Longleftrightarrow \mathscr{O}_{S}\big[f\big](\zeta)= 0,    ~| \zeta |\geq B\pi
		\\
		&\Longleftrightarrow \mathscr{O}_{S}\big[f\big]\left(\tfrac{\zeta}{2}\right) = 0, ~ \left| \tfrac{\zeta}{2} \right|\geq B\pi
		\\
		&\Longleftrightarrow \mathscr{O}_{S}\big[f\big]\left(\tfrac{\zeta}{2}\right) = 0, ~  | \zeta |\geq 2B\pi
		\\
		&\Longleftrightarrow
		\tfrac{1}{2} \exp\left\{ \tfrac{i}{2B} \left( -\zeta(Dp - Bq)+  \tfrac{3}{4}\zeta^2 D  \right) \right\}\mathscr{O}_{S} \left[f\right]\left(\tfrac{\zeta}{2}\right)=0, ~ | \zeta |\geq 2B\pi
		\\
		&\Longleftrightarrow\exp\left\{\tfrac{i}{2B}\left(3Ax^2+2xp\right)\right\} f(2x)  \in  \mathscr{V}_{S}^1
	\end{align*}
	
	\parindent=8mm\vspace{.1in}	
	In general, for $T = 1/2^k$ and $\Omega_S = 2^k B\pi$, the space of special affine band-limited signals is defined as	
	$$\mathscr{V}_{S}^k=\left\{f \in \ell^2\mathbb{(R)}: \mathscr{O}_{S}\big[f\big](\zeta )= 0, ~ | \zeta |\geq 2^kB\pi\right\}$$
	and $\forall f \in 	\mathscr{V}_{S}^k$, the function admits the reconstruction formula
	\begin{align*}
		f(x)&=\sum_{n \in \mathbb{Z}}f\left(\tfrac{n}{2^k}\right)\exp\left\{-\tfrac{i}{2B}\left(A\left(x^2-{\big(\tfrac{n}{2^k}\big)}^2\right)+2p\big(x-\tfrac{n}{2^k}\big)\right)\right\}\text{sinc}\left(2^kx-n\right)
		\\
		&=\frac{1}{2^{{k}/{2}}}\sum_{n \in \mathbb{Z}} f\left(\tfrac{n}{2^k}\right) \varphi_{S,k,n}(x),
	\end{align*}	
	where
	\begin{align}\label{S3E3}
		\varphi_{S,k,n}(x)=2^{{{k}/{2}}}\exp\left\{-\tfrac{i}{2B}\left(A\left(x^2-{\big(\tfrac{n}{2^k}\big)}^2\right)+2p\big(x-\tfrac{n}{2^k}\big)\right)\right\} \text{sinc}\left(2^kx-n\right) .
	\end{align}  
	Moreover, the family $\left\{\varphi_{S,k,n} : n \in \mathbb{Z}\right\}$ forms an orthonormal basis for the space $\mathscr{V}_{S}^k$.
	
	\parindent=8mm\vspace{.1in}	
	We are now in a position to establish the formal definition of a SAMRA within $L^2(\mathbb R)$.
	\begin{definition}\label{S3D1}
		Given a real augmented matrix $S = [M\mid \Lambda]$, a corresponding SAMRA is defined as a collection $\left\{\mathscr{V}_S^k : k \in \mathbb{Z}\right\}$ of closed subspaces of $L^2(\mathbb{R})$ that satisfy the following properties:
		\begin{flushleft}
			\rm(a) $\mathscr{V}_{S}^k \subset  \mathscr{V}_{S}^{k+1}$, for all $k\in\mathbb Z$;\\
			\rm(b)	$ f(x) \in \mathscr{V}_{S}^k$ if and only if $e^{i\left(3Ax^2+2xp\right)/2B} f(2x)  \in  \mathscr{V}_{S}^{k+1} $,  for all $k\in\mathbb Z$;\\
			\rm(c) $\bigcap\limits_{k \in \mathbb{Z}}\mathscr{V}_{S}^k = \left\{0\right\}$ and $\overline{\bigcup_{k\in \mathbb{Z}}\mathscr{V}_{S}^k}= L^2(\mathbb R)$;\\
			\rm(d) There exist a function $\varphi \in L^2(\mathbb R) $ such that 
			$\big\{\varphi_{S,0,n}(x) = e^{-i\left(A\left(x^2-n^2\right)+2p(x-n)\right)/2B} f(n) $ $\text{sinc}(x-n): n \in \mathbb{Z}\big\}$
			is an orthonormal basis of $\mathscr{V}_{S}^0$.
		\end{flushleft}  
	\end{definition}
	
	\begin{remark}\label{S3R2}
		It is worth noting that Definition \ref{S3D1} generalizes several existing MRAs found in the literature. Specifically, when $S = [M \mid 0]$, it reduces to the linear canonical MRA as described in \cite{LCT}; for $M = [M_\theta \mid 0]$, where $M_\theta=\left[\begin{smallmatrix} \cos\theta&\sin\theta\\-\sin\theta&\cos\theta\end{smallmatrix}\right], \theta\neq n\pi, n\in\mathbb Z$, it yields the fractional MRA discussed in \cite{15}; and in the special case $M = [I \mid 0]$, Definition \ref{S3D1} coincides with the classical MRA framework.
	\end{remark}	
	
	The following lemma will be employed in establishing several results throughout this article.
	
	\begin{lemma}\label{S3L3}
		The family 	$\left\{\varphi_{S,0,n}(x):{n \in \mathbb{Z}}\right\}$, as given in \eqref{S3E1},  forms an orthonormal system in  $L^2(\mathbb R)$ if and only if
		\begin{align}\label{S3E4}
			\sum_{k \in \mathbb{Z}}{\left|\mathcal {F}\big[\varphi \big]\left(\zeta+2k\pi\right)\right|}^2 = 1.
		\end{align}
	\end{lemma}
	
	\begin{proof}
		By applying the SAFT to the family defined in \eqref{S3E1}, we obtain
		\begin{align*}
			&\mathscr{O}_{S}\left[{\varphi}_{S,0,n}\right](\zeta)
			\\
			&=\int_{\mathbb R}{\varphi}_{S,0,n}(x)\,\mathcal{K}_S(x,\zeta)\,dx
			\\
			&=\tfrac{1}{\sqrt{2i\pi B
			}}\int_{\mathbb R}	\exp\left\{-\tfrac{i}{2B}\left(A(x^2-n^2)+2p(x-n)\right)\right\} \varphi{(x-n)}
			\\
			&\quad\times	\exp\left\{\tfrac{i}{2B}\left(Ax^2+2x(p-\zeta)-2\zeta(Dq-Bq)+D\left(\zeta^2 + p^2\right)\right)\right\}dx
			\\
			&=\tfrac{1}{\sqrt{2i\pi B
			}}\int_{\mathbb R}	\exp\left\{-\tfrac{i}{2B}\left(-An^2-2np+2x\zeta+2\zeta(Dp-Bq)+D\left(\zeta^2 + p^2\right)\right)\right\} \varphi(x-n) \,dx
			\\
			&=\tfrac{1}{\sqrt{2i\pi B}}\exp\left\{-\tfrac{i}{2B}\left(-An^2-2np+2x\zeta+2\zeta(Dp-Bq)+D\left(\zeta^2 + p^2\right)\right)\right\}
			e^{-i\zeta n/B}\int_{\mathbb R} e^{-ix\zeta/B}\varphi(x) \, dx
			\\
			&=\tfrac{1}{\sqrt{2i\pi B}}\exp\left\{-\tfrac{i}{2B}\left(-An^2-2np+2x\zeta+2\zeta(Dp-Bq)+D\left(\zeta^2 + p^2\right)\right)\right\}
			e^{-i\zeta n/B}  \,\mathcal {F}\big[\varphi \big]\Big(\tfrac{\zeta}{B}\Big).
		\end{align*}	
		By virtue of Parseval's formula \eqref{1}, we have
		\begin{align*}
			&\big\langle \varphi_{S,0,n}, \varphi_{S,0,\ell}\big \rangle 
			\\ 
			&\qquad=  \Big\langle \mathscr{O}_{S}\left[\varphi_{S,0,n}\right], \mathscr{O}_{S}\left[\varphi_{S,0,\ell}\right] \Big\rangle 
			\\
			&\qquad=\int_{\mathbb R} \mathscr{O}_{S}\left[\varphi_{S,0,n}\right](\zeta)\,\overline{\mathscr{O}_{S}\left[\varphi_{S,0,\ell}\right](\zeta)}\,d\zeta
			\\
			&\qquad=\int_{\mathbb R}	\tfrac{1}{{2\pi B}}\exp\left\{-\tfrac{i}{2B}\left(-An^2-2np+2x\zeta+2\zeta(Dp-Bq)+D\left(\zeta^2 + p^2\right)\right)\right\}
			\\
			&\qquad\qquad\times  e^{\tfrac{-i\zeta n}{B}}\mathcal {F}\big[\varphi\big]\Big(\tfrac{\zeta}{B}\Big) \exp\left\{\tfrac{i}{2B}\left(-A\ell^2-2lp+2x\zeta+2\zeta(Dp-Bq)+D\left(\zeta^2 + p^2\right)\right)\right\}
			\\
			&\qquad\qquad\times e^{i\zeta n/B} \,
			\overline{\mathcal {F}\big[\varphi\big]\Big(\tfrac{\zeta}{B}\Big)} \, d\zeta
			\\
			&\qquad=\tfrac{1}{2\pi B}\int_{\mathbb R}\exp\left\{-\tfrac{i}{2B}\left(A(\ell^2-n^2)-2p(\ell-n)\right)\right\} e^{-i\zeta (\ell-n)/B} \left|\mathcal {F}\big[\varphi\big]\Big(\tfrac{\zeta}{B}\Big)\right|^2d\zeta
			\\
			&\qquad=\tfrac{1}{2\pi }\exp\left\{-\tfrac{i}{2B}\left(A(\ell^2-n^2)-2p(\ell-n)\right)\right\}\int_{\mathbb R} e^{-i\zeta (\ell-n)} \left|\mathcal {F}\big[\varphi\big](\zeta)\right|^2d\zeta.
		\end{align*}
		Since $\left\{\varphi_{S,0,n}:n \in \mathbb{Z} \right\}$ constitutes an orthonormal basis for $\mathscr{V}_{S}^0$, it follows that
		\begin{align*}
			\tfrac{1}{2\pi }\exp\left\{-\tfrac{i}{2B}\left(A(\ell^2-n^2)-2p(\ell-n)\right)\right\}\int_{\mathbb R} e^{-i\zeta (\ell-n)} \left|\mathcal {F}\big[\varphi\big](\zeta)\right|^2d\zeta =\delta_{n,l}.
		\end{align*}		
		Setting  $l-n=m$ gives
		\begin{align*}
			\tfrac{1}{2\pi } \exp\left\{-\tfrac{i}{2B}\left(A\left((m+n)^2-n^2\right)\right)-2pm\right\}\int_{\mathbb R} e^{-i\zeta m}\left|\mathcal {F}\big[\varphi\big](\zeta)\right|^2d\zeta =\delta_{n,m+n},
		\end{align*}	
		which implies				
		\begin{align*}
			\tfrac{1}{2\pi} \int_{\mathbb R} e^{-i\zeta m}\left|\mathcal {F}\big[\varphi\big](\zeta)\right|^2d\zeta =\delta_{m},~ m \in \mathbb{Z}.
		\end{align*}
		Equivalently,	
		\begin{align*}
			\tfrac{1}{2\pi}\sum_{k \in \mathbb{Z}}\int_{2k\pi}^{2(k+1)\pi} e^{-i\zeta m}\left|\mathcal {F}\big[\varphi\big](\zeta)\right|^2d\zeta =\delta_{m},
		\end{align*}
		or, changing variables:
		\begin{align*}
			\tfrac{1}{2\pi} \int_{0}^{2\pi} e^{-im(\eta+2k\pi)}\sum_{k \in \mathbb{Z}}\left|\mathcal {F}\big[\varphi\big](\eta+2k\pi)\right|^2d\eta =\delta_{m},
		\end{align*}
		which simplifies to 
		\begin{align}\label{S3E5}
			\tfrac{1}{2\pi} \int_{0}^{2\pi} e^{-im\eta}\sum_{k \in \mathbb{Z}}\left|\mathcal {F}\big[\varphi\big]\left({\eta+2k\pi}\right)\right|^2d\eta =\delta_{m}.
		\end{align}		
		Let $F(\eta) = \sum_{k \in \mathbb{Z}} \left|\mathcal{F}[\varphi](\eta + 2k\pi)\right|^2$. It can be readily verified that $F(\eta)$ is a $2\pi$-periodic function. Therefore, from \eqref{S3E5}, it follows that		
		\begin{align*}
			\tfrac{1}{2\pi}\int_{0}^{2\pi} e^{-im\eta}{F}\big(\eta\big)d\eta =\delta_{m},
		\end{align*}
		which implies that $F(\eta) =1$. This establishes the necessary condition.
		
		\parindent=8mm\vspace{.1in}		
		Conversely, suppose that $\sum_{k \in \mathbb{Z}}\left|\mathcal {F}\big[\varphi\big]\left({\zeta+2k\pi}\right)\right|^2=1$. Then, we have
		\begin{align*}
			\big\langle \varphi_{S,0,n}, \varphi_{S,0,\ell} \big\rangle 
			&=\tfrac{1}{2\pi }\exp\left\{-\tfrac{i}{2B}\left(A(\ell^2-n^2)-2p(\ell-n)\right)\right\} \int_{\mathbb R} e^{-i\zeta (\ell-n)} \left|\mathcal {F}\big[\varphi\big](\zeta)\right|^2d\zeta
			\\
			&= \tfrac{1}{2\pi }\exp\left\{-\tfrac{i}{2B}\left(A(\ell^2-n^2)-2p(\ell-n)\right)\right\}\sum_{k \in \mathbb{Z}}\int_{2k\pi}^{2(k+1)\pi} e^{-i\zeta (\ell-n)}\left|\mathcal {F}\big[\varphi\big](\zeta)\right|^2d\zeta 
			\\
			&= \tfrac{1}{2\pi }\exp\left\{-\tfrac{i}{2B}\left(A(\ell^2-n^2)-2p(\ell-n)\right)\right\}\int_{0}^{2\pi} e^{-i\zeta (\ell-n)}\sum_{k \in \mathbb{Z}}\left|\mathcal {F}\big[\varphi\big]\left({\zeta+2k\pi}\right)\right|^2d\zeta
			\\
			&= \tfrac{1}{2\pi }\exp\left\{-\tfrac{i}{2B}\left(A(\ell^2-n^2)-2p(\ell-n)\right)\right\}\int_{0}^{2\pi} e^{-i\zeta (\ell-n)} 	d\zeta 
			\\
			&=\delta_{n,l}.
		\end{align*}
		This concludes the proof.
	\end{proof}
	
	 Since $\varphi_{S,0,0}(x) \in \mathscr{V}_{S}^0  \subset   \mathscr{V}_{S}^1$,  and  $\left\{\varphi_{S,1,n}(x):  n \in \mathbb{Z} \right\} $ is an orthonormal basis of  $\mathscr{V}_{S}^1$, there must exist a sequence  $\left\{h_n:n \in \mathbb{Z}\right\}$ such that
	
	\begin{align}\label{S3E6}
		\varphi_{S,0,0}(x)&=\sum_{n \in \mathbb{Z}}h_n \, \varphi_{S,1,n}(x)\notag
		\\
		&  =\sqrt{2} \sum_{n \in \mathbb{Z}}h_n\exp\left\{-\tfrac{i}{2B}\left(A(x^2-\tfrac{n^2}{4})+  2p\left(x-\tfrac{n}{2}\right)\right)\right\} \varphi{\left(2x-n\right)},
	\end{align}
	where 
	\begin{align}\label{V1}
		h_n = \sqrt{2}\int_{\mathbb R}\exp\left\{-\tfrac{i}{2B}\left(\tfrac{An^2}{4}+2pn\right)\right\}  \varphi{(x)} \, \overline{\varphi\left(2x-n\right)} \, dx.
	\end{align}
	Equation \eqref{S3E6} is referred to as the special affine refinement equation, whereas \eqref{V1} is known as the special affine low-pass filter. Applying SAFT on \eqref{S3E6} gives
	\begin{align}\label{V2}
		\mathcal{F}\big[\varphi\big]\Big(\tfrac{\zeta}{B}\Big) &= \sqrt{2}\,\sum_{n \in \mathbb{Z}}h_n\exp\left\{\tfrac{i}{2B}\big(\tfrac{An^2}{4}+pn\big)\right\}\int_{\mathbb R}\varphi(2x-n) \, e^{-i\zeta x/B} dx \notag
		\\
		&=S_0\Big(\tfrac{\zeta}{2B}\Big) \mathcal {F}\big[\varphi\big]\Big(\tfrac{\zeta}{2B}\Big),
	\end{align}
	 where 
	 \begin{align}\label{S3E8}
	 	S_0(\zeta)= \tfrac{1}{\sqrt{2}}\sum_{n \in \mathbb{Z}}h_n \, e^{-i n \zeta}\exp\left\{\tfrac{i}{2B}\left(\tfrac{An^2}{4}+pn\right)\right\}.
	 \end{align}
	 Clearly, $S_0(\zeta)$ is $2\pi$-periodic, that is, $S_0(\zeta + 2\pi) = S_0(\zeta)$. More generally, we have
	\begin{align}\label{S3E9}
		S_0(\zeta + 2k\pi) = S_0(\zeta), ~ \forall\,k \in \mathbb{Z}.
	\end{align}	
	Since $\left\{\varphi_{S,0,n}(x):n \in \mathbb{Z}\right\}$ forms an orthonormal basis of $\mathscr{V}_{S}^0$, it follows from  Lemma \ref{S3L3} that
	\begin{align}\label{S3E10}
		1&=\sum_{k \in \mathbb{Z}}{\left|\mathcal {F}\big[\varphi \big]\left(\zeta+2k\pi\right)\right|}^2 \notag 
		\\
		&=\sum_{k \in \mathbb{Z}}\left|S_0\left(\tfrac{\zeta}{2}+k\pi\right)\right|^2\left|\mathcal{F}\left(\tfrac{\zeta}{2}+k\pi\right)\right|^2 \notag 
		\\
		&=\sum_{k \in \mathbb{Z}}\left|S_0\left(\tfrac{\zeta}{2}+2k\pi\right)\right|^2\left|\mathcal{F}\left(\tfrac{\zeta}{2}+2k\pi\right)\right|^2+\sum_{k \in \mathbb{Z}}\left|S_0\left(\tfrac{\zeta}{2}+(2k+1)\pi\right)\right|^2\left|\mathcal{F}\left(\tfrac{\zeta}{2}+(2k+1)\pi\right)\right|^2 \notag 
		\\
		&=\left|S_0\left(\tfrac{\zeta}{2}\right)\right|^2\sum_{k \in \mathbb{Z}}\left|\mathcal{F}\left(\tfrac{\zeta}{2}+2k\pi\right)\right|^2+\left|S_0\left(\tfrac{\zeta}{2}+\pi\right)\right|^2\sum_{k \in \mathbb{Z}}\left|\mathcal{F}\left(\tfrac{\zeta}{2}+2k\pi+\pi\right)\right|^2 \notag 
		\\
		&=\left|S_0\left(\tfrac{\zeta}{2}\right)\right|^2+\left|S_0\left(\tfrac{\zeta}{2}+\pi\right)\right|^2.
	\end{align}
	
	\parindent=8mm\vspace{.05in}		
	Given an orthogonal MRA $\left\{\mathscr{V}_S^n\right\}_{n \in \mathbb{Z}}$, we define an associated sequence of closed subspaces $\left\{\mathscr{W}_S^n: n \in \mathbb{Z}\right\}$  of $L^2(\mathbb{(R)}$ by the relation $\mathscr{V}_S^{k+1} = \mathscr{V}_S^k \oplus\mathscr{W}_S^k, k \in\mathbb{Z}$. By definition, these subspaces preserve the scaling behavior of the multiresolution spaces   $\left\{\mathscr{V}_S^k : k \in \mathbb{Z}\right\}$, namely:
	\begin{align}\label{V3}
		f(x) \in\mathscr{W}_{S}^k \Longleftrightarrow \exp\left\{\tfrac{i}{2B}\left(3Ax^2+2xp\right)\right\} f(2x)  \in  \mathscr{W}_{S}^{k+1}, ~k \in \mathbb{Z}.
	\end{align}
	Moreover, the subspaces $ \mathscr{W}_S^k $	are mutually orthogonal, leading to the following decomposition formula:
	\begin{align}\label{S3E11}
		L^2(\mathbb R) = \bigoplus_{k \in \mathbb{Z}}\mathscr{W}_S^k
	\end{align}
	It is important to note that condition \eqref{S3E11} guarantees that an orthonormal basis for 
	$ L^2(\mathbb{R}) $ can be constructed by identifying an orthonormal basis for each subspace 
	$ \mathscr{W}_S^k $. Furthermore, condition \eqref{S3E9} implies that constructing an orthonormal basis for a single subspace $ \mathscr{W}_S^k $ is sufficient to generate the entire special affine wavelet basis across all scales. Therefore, the main objective reduces to the construction of a mother wavelet $ \psi \in \mathscr{W}_S^0 $ such that the collection 
	$ \{ \psi_{S,0,k}:k \in \mathbb{Z} \} $ forms an orthonormal basis for $\mathscr{W}_S^0 $.

	\parindent=8mm\vspace{.1in}
	Suppose $\psi_{S,0,0}(x) \in\mathscr{W}_S^0\subset \mathscr{V}_S^1 $. Then, there exists a sequence $\{d_k: {k \in \mathbb{Z}}\}$ such that 
	\begin{align}\label{S3E13}
		\psi_{S,0,0}(x) =\sqrt{2}\sum_{k \in \mathbb{Z}}d_k\exp\left\{-\tfrac{i}{2B}\left(A\left(x^2-\tfrac{k^2}{4}\right)+2p\left(x-\tfrac{k}{2}\right)\right)\right\}\varphi{\left(2x-k\right)}.
	\end{align}
	Equation \eqref{S3E13} is referred to as the special affine wavelet equation. Taking the SAFT of both sides of \eqref{S3E13}, we obtain
	\begin{align}\label{S3E14}
		\mathcal{F}\big[\psi\big]\Big(\tfrac{\zeta}{B}\Big) = S_1\left(\tfrac{\zeta}{2B}\right)  \mathcal{F}\big[\varphi\big]\left(\tfrac{\zeta}{2B}\right),
	\end{align}
	where  
	\begin{align}\label{S3E15}
		S_1(\zeta)= \tfrac{1}{\sqrt{2}}\sum_{k \in \mathbb{Z}} d_k\,  e^{-i k \zeta} \exp\left\{\tfrac{i}{2B}\left(\tfrac{Ak^2}{4}+pk\right)\right\}
	\end{align}
	is a $2\pi-$periodic function.

	\parindent=8mm\vspace{.1in}	
	Since $ \mathscr{W}_{S}^0 $ and $ \mathscr{V}_{S}^0 $ are orthogonal in $ \mathscr{V}_{S}^1 $, we have
	\begin{align}\label{V4}
		0&=\big\langle \varphi_{S,0,k}, \psi_{S,0,\ell}\big \rangle \notag
		\\
		&=\int_{\mathbb R}\mathscr{O}_{S}\big[\varphi_{S,0,k}\big](\zeta) \,\overline{\mathscr{O}_{S}\big[\psi_{S,0,\ell}\big](\zeta)} \, d\zeta  \notag
		\\
		&=\frac{1}{2\pi B} \exp\left\{-\tfrac{i}{2B}\left(A\left(\ell^2-k^2\right)-2p(\ell-k)\right)\right\} \int_{\mathbb R}  e^{-i\zeta (\ell-k)/B} \mathcal {F}\big[\varphi\big]\Big(\tfrac{\zeta}{B}\Big)\,\overline{\mathcal {F}\left(\psi\right)\Big(\tfrac{\zeta}{B}\Big)} \, d\zeta  \notag
		\\
		&=\frac{1}{2\pi }\exp\left\{-\tfrac{i}{2B}\left(A\left(\ell^2-k^2\right)-2p(\ell-k)\right)\right\}\int_{\mathbb R}\, e^{-i\zeta (\ell-k)}\mathcal {F}\big[\varphi\big](\zeta)\overline{\mathcal {F}\big[\psi\big](\zeta)} \, d\zeta  \notag
		\\
		&=\int_{\mathbb R} e^{-i\zeta (\ell-k)} \mathcal {F}\big[\varphi\big](\zeta)\overline{\mathcal {F}\big[\psi\big](\zeta)}\, d\zeta  \notag
		\\
		&=\sum_{m \in \mathbb{Z}}\int_{4m\pi}^{4(m+1)\pi}\, e^{-i\zeta (\ell-k)} \mathcal {F}\big[\varphi\big](\zeta)\overline{\mathcal {F}\big[\psi\big](\zeta)}\,  d\zeta.
	\end{align}
	Substituting equations \eqref{V2} and \eqref{S3E14} into \eqref{V4}, we obtain	
	\begin{align}\label{S3E19}
		0&=\sum_{m \in \mathbb{Z}}\int_{0}^{4\pi}\exp\left\{{-i(\zeta+4m\pi) (\ell-k)}\right\} S_0\left(\tfrac{\zeta}{2}+2m\pi\right)  \overline{ S_1\left(\tfrac{\zeta}{2}+2m\pi\right)}            \,        {\left|\mathcal{F}\big[\varphi\big]\left(\tfrac{\zeta}{2}+2m\pi\right) \right|}^2d\zeta \notag
		\\
		&=\sum_{m \in \mathbb{Z}}\int_{0}^{4\pi}\exp\left\{{-i(\zeta+4m\pi) (\ell-k)}\right\} S_0\left(\tfrac{\zeta}{2}\right)  \overline{ S_1\left(\tfrac{\zeta}{2}\right)} \,    {\left|\mathcal{F}\big[\varphi\big]\left(\tfrac{\zeta}{2}+2m\pi\right) \right|}^2d\zeta \notag
		\\
		&=\sum_{m \in \mathbb{Z}}\int_{0}^{2\pi}\, e^{-i\zeta (\ell-k)} S_0\left(\tfrac{\zeta}{2}\right)  \overline{ S_1\left(\tfrac{\zeta}{2}\right)} \,   {\left|\mathcal{F}\big[\varphi\big]\left(\tfrac{\zeta}{2}+2m\pi\right) \right|}^2d\zeta \notag
		\\
		&\qquad\qquad+\sum_{m \in \mathbb{Z}}\int_{2\pi}^{4\pi}\, e^{-i\zeta (\ell-k)} S_0\left(\tfrac{\zeta}{2}\right)  \overline{ S_1\left(\tfrac{\zeta}{2}\right)} \,   {\left|\mathcal{F}\big[\varphi\big]\left(\tfrac{\zeta}{2}+2m\pi\right) \right|}^2d\zeta \notag
		\\
		&=\int_{0}^{2\pi}\, e^{-i\zeta (\ell-k)} S_0\left(\tfrac{\zeta}{2}\right)  \overline{ S_1\left(\tfrac{\zeta}{2}\right)}    {\sum_{m \in \mathbb{Z}}\left|\mathcal{F}\big[\varphi\big]\left(\tfrac{\zeta}{2}+2m\pi\right) \right|}^2d\zeta \notag
		\\
		&\qquad\qquad+\sum_{m \in \mathbb{Z}}\int_{0}^{2\pi}\, e^{-i\zeta (\ell-k)} S_0\left(\tfrac{\zeta}{2}+\pi\right)  \overline{ S_1\left(\tfrac{\zeta}{2}+\pi\right)}    {\sum_{m \in \mathbb{Z}}\left|\mathcal{F}\big[\varphi\big]\left(\tfrac{\zeta}{2}+(2m+1)\pi\right) \right|}^2d\zeta \notag
		\\
		&=	\int_0^{2\pi} \, e^{-i(\ell-k)\zeta} \left[ S_0\left( \tfrac{\zeta}{2} \right)S_1\left( \tfrac{\zeta}{2} \right) + S_0\left( \tfrac{\zeta}{2} + \pi \right)S_1\left( \tfrac{\zeta}{2} + \pi \right) \right] d\zeta.
		\end{align}
	Therefore, we conclude that
	\begin{align}\label{S3E20}
		S_0\left( \tfrac{\zeta}{2} \right)S_1\left( \tfrac{\zeta}{2} \right) + S_0\left( \tfrac{\zeta}{2} + \pi \right)S_1\left( \tfrac{\zeta}{2} + \pi \right)  = 0.
	\end{align}
	Equation  \eqref{S3E20} can be written in the matrix form as
	\begin{align*}
		\mathscr H(\zeta) \mathscr H^*(\zeta) = I_{2 \times 2},
	\end{align*}
	where $\mathscr H^* $ denotes the conjugate transpose of $\mathscr H, I_{2 \times 2} $ is the $ 2 \times 2 $ identity matrix, and $\mathscr H$ is defined by
	\begin{align*}
		\mathscr	H(\zeta) = \begin{bmatrix}
			S_0(\zeta) & S_0(\zeta + \pi) \\
			S_1(\zeta) & S_1(\zeta + \pi)
		\end{bmatrix}.
	\end{align*}
	
	\parindent=8mm\vspace{.05in}	
	Since ${S_0(\zeta)}$ and ${S_0(\zeta + \pi)}$ cannot simultaneously vanish on a set of positive measure due to their orthogonality property, there exists a $2\pi$-periodic function $\rho(\zeta)$ such that
	\begin{align}\label{V5}
		\big( S_1(\zeta), S_1(\zeta + \pi) \big) = \big( \rho(\zeta)\,\overline{S_0(\zeta + \pi)}, -\rho(\zeta)\,\overline{S_0(\zeta )}\,\big).
	\end{align}
	Since
	\begin{align*}
		S_1(\zeta)=\rho(\zeta)\,\overline{S_0(\zeta + \pi)},
	\end{align*}	
	we have		
	\begin{align}\label{V7}	
		 S_1(\zeta+\pi)&=\rho(\zeta+\pi) \,\overline{S_0(\zeta + 2\pi)}\notag\\
		 &=\rho(\zeta+\pi)\,\overline{S_0(\zeta)}.
	\end{align}
	Therefore, from equation \eqref{V5}, we have
	\begin{align*}
		-\rho(\zeta)\,\overline{S_0(\zeta)} = \rho(\zeta + \pi)\,\overline{S_0(\zeta)},
	\end{align*}
	which implies that
	\begin{align}\label{V6}
		\rho(\zeta) +\rho(\zeta + \pi)=0.
	\end{align}
	Equation \eqref{V6} suggests that the function $\rho(\zeta)$ is $2\pi$-periodic. Therefore, it can be represented using its Fourier series expansion as:
	\begin{align*}
		\rho(\zeta) = \sum_{k \in \mathbb{Z}} c_k \, e^{-ik\zeta},
	\end{align*}
	where
	\begin{align*}
		c_k &= \frac{1}{2\pi} \int_0^{2\pi} \rho(\zeta) \, e^{ik\zeta}d\zeta
		\\
		&=\frac{1}{2\pi} \left[ \int_0^{\pi} \rho(\zeta) \, e^{ik\zeta} d\zeta + \int_{\pi}^{2\pi} \rho(\zeta) \, e^{ik\zeta} \, d\zeta \right]
		\\
		&= \frac{1}{2\pi} \left[ \int_0^{\pi} \rho(\zeta) \, e^{ik\zeta} d\zeta + \int_0^{\pi} \rho(\zeta + \pi) \, e^{ik(\zeta + \pi)}d\zeta \right]
		\\
		&=\frac{1}{2\pi} \left[ \int_0^{\pi} \rho(\zeta) \, e^{ik\zeta}d\zeta - (-1)^k \int_0^{\pi} \rho(\zeta) \, e^{ik\zeta}d\zeta \right]
		\\
		&= \tfrac{1}{2\pi} \left( 1 - (-1)^k \right) \int_0^{\pi} \rho(\zeta) \, e^{ik\zeta}  d\zeta.
	\end{align*}
	From the above discussion, we conclude that $c_k = 0$ for all even integers $k = 2m$, where $m \in \mathbb{Z}$. Hence, we can write: 
	\begin{align*}
		\rho(\zeta) &= \sum_{\ell \in \mathbb{Z}} c_{\ell+1} \, e^{-i(2l+1)\zeta}
		\\
		&= e^{-i\zeta} \,\gamma(2\zeta),
	\end{align*}
	where $\gamma(2\zeta)=\sum_{\ell \in \mathbb{Z}} c_{2\ell+1} \, e^{-2i\ell\zeta}$. 
	
	\parindent=8mm\vspace{.1in}
	Now, using equation \eqref{V7}, we obtain
	\begin{align*}
		S_1(\zeta)&=\rho(\zeta)\overline{S_0(\zeta + \pi)}
		\\
		&= e^{-i\zeta} \, \gamma(2\zeta)\,\overline{S_0(\zeta + \pi)}.
	\end{align*}	
	In particular, setting $\gamma(2\zeta) = 1$, we obtain
	\begin{align}\label{V8}
		S_1(\zeta)=\, e^{-i\zeta} \, \overline{S_0(\zeta + \pi)}.
	\end{align}	
	Substituting equations \eqref{S3E8} and \eqref{S3E15} into \eqref{V8}, we obtain
	\begin{align*}
		& \tfrac{1}{\sqrt{2}} \sum_{n \in \mathbb{Z}} d_n\,e^{-i n \zeta} \exp\left\{\tfrac{i}{2B}\left(\tfrac{An^2}{4}+pn\right)\right\} =\tfrac{e^{-i\zeta}}{\sqrt{2}} \, \sum_{n \in \mathbb{Z}} (-1)^n \, \overline{h_n}\,e^{i n \zeta} \exp\left\{-\tfrac{i}{2B}\left(\tfrac{An^2}{4}+pn\right)\right\}.
	\end{align*}
	Equivalently,
	\begin{align*}
		\sum_{n \in \mathbb{Z}}d_n\exp\left\{\tfrac{i}{2B}\left(\tfrac{An^2}{4}+pn\right)\right\} e^{-i\zeta(n-k)} = \sum_{n \in \mathbb{Z}}(-1)^n \, \overline{h_n} \exp\left\{-\tfrac{i}{2B}\left(\tfrac{An^2}{4}+pn\right)\right\} e^{-i\zeta(1-n-k)}.
	\end{align*}
	On integrating both sides, we obtain	
	\begin{align*}
		&\sum_{n \in \mathbb{Z}}d_n\exp\left\{\tfrac{i}{2B}\left(\tfrac{An^2}{4}+pn\right)\right\}\int_{\mathbb R} e^{-i\zeta(n-k)} d\zeta=\sum_{n \in \mathbb{Z}}(-1)^n \, \overline{h_n} \exp\left\{-\tfrac{i}{2B}\left(\tfrac{An^2}{4}+pn\right)\right\}\int_{\mathbb R}e^{-i\zeta(1-n-k)}d\zeta,
	\end{align*}	
	which implies that 		
	\begin{align}\label{V9}
		\sum_{n \in \mathbb{Z}}d_n\exp\left\{\tfrac{i}{2B}\left(\tfrac{An^2}{4}+pn\right)\right\}\delta_{n-k}= \sum_{n \in \mathbb{Z}} (-1)^n \, \overline{h_n} \exp\left\{-\tfrac{i}{2B}\left(\tfrac{An^2}{4}+pn\right)\right\}\delta_{1-n-k}.
	\end{align}	
	The wavelet coefficients ${d_k}$ can be derived from equation \eqref{V9} and are given by	
	\begin{align}\label{S3E25}
		 d_k=(-1)^{1-k}\,\overline{h_k}\exp\left\{-\tfrac{i}{2B}\left(\tfrac{A\left((1-k)^2+k^2\right)}{4}+p\right)\right\}.
	\end{align}

	\parindent=8mm\vspace{.1in}
	The above discussion leads to the formulation of the following theorem.
	\begin{theorem}\label{S3T1}
		If $\{V^{n}_S:n \in \mathbb{Z}\}$ denoted the special affine MRA associated with the scaling function $\varphi$, then there exists a function $\psi$ such that
		\begin{align}\label{S3E26}
			\psi_{S,0,0}(x) =\sqrt{2} \, \sum_{k \in \mathbb{Z}}d_k\exp\left\{-\tfrac{i}{2B}\left(A \left(x^2-\tfrac{k^2}{4}\right)+2p\left(x-\tfrac{k}{2}\right)\right)\right\}  \varphi{\left(2x-k\right)},
		\end{align}
		where the wavelet coefficients $d_k$ are given by \eqref{S3E25}, and the scaling coefficients $h_k$ are given by \eqref{V1}.  Consequently, the system $\{ \psi_{S,k,n} \,:\, k,n \in \mathbb{Z} \}$	forms an orthonormal basis of $L^2(\mathbb R)$.
	\end{theorem}
	
	Towards the end of this section, we provide illustrative examples to enhance comprehension of the orthonormal special affine wavelet construction.
	
	\begin{example}\label{S3EX5}
		Let $\varphi(x)=\mathrm{\text{sinc}}(\pi x)$ be the scaling function for the SAMRA  $\{\mathscr{V}^{S}_n=\varphi_{S,0,n}(x):\,n \in \mathbb{Z}\}$, where $ \varphi_{S,0,n}(x) = e^{-i\left(A(x^2 - n^2) + 2p(x- n)\right)/2B} \, \text{sinc}(x - n)$. Then, by virtue of equations \eqref{V1} and \eqref{S3E25}, the scaling coefficients $h_n$ and the wavelet coefficients $d_n$ are given by		
		\begin{align*}
			h_n =
			\begin{cases}
				\tfrac{1}{\sqrt{2}}, & n = 0 \\[1.5ex]
				\tfrac{\sqrt{2}}{\pi n} \,
				\exp\left\{-\tfrac{i}{2B} \left( \tfrac{A n^2}{4} + pn \right) \right\}  
				\sin\left( \tfrac{n\pi}{2} \right), & n \neq 0
			\end{cases}
		\end{align*}
		and
		\begin{align*}
			d_n =
			\begin{cases}
				\tfrac{1}{\sqrt{2}} \, \exp\left\{-\tfrac{i}{2B}\left( \tfrac{A}{4} + p \right)\right\}, & n = 1 \\[1.5ex]
				\tfrac{\sqrt{2}}{\pi(n - 1)} \, (-1)^{-n}  \cos\left( \tfrac{n\pi}{2} \right) \, 
				\exp\left\{- \tfrac{i}{2B} \left(\tfrac{A n^2}{4} + p n \right)\right\}, & n \neq 1.
			\end{cases}
		\end{align*}
	Thus, the special affine wavelet corresponding to the scaling function $\text{sinc}(x)$ is given by
	\begin{align}\label{V10}
		\psi_{S,0,0}(x) &=
		\sqrt{2} ~ d_1  \exp\left\{-\tfrac{i}{2B}\left(A \left(x^2 - \tfrac{1}{4}\right) + 2p \left(x- \tfrac{1}{2}\right)\right)\right\} \varphi(2x - 1)	\notag	
		\\
		&\qquad+\sqrt{2} \, \sum_{\substack{n \in \mathbb{Z}, n \neq 1}}
		d_n \, \exp\left\{-\tfrac{i}{2B} \left( A \left( x^2 - \tfrac{n^2}{4} \right) + 2p\left( x - \tfrac{n}{2} \right) \right)\right\} \varphi(2x - n) \notag
		\\
		&=\exp\left\{ -\tfrac{i}{2B} \left( \tfrac{A}{4} + p \right) \right\} 
		\exp\left\{ -\tfrac{i}{2B} \left( A \left( x^2 - \tfrac{1}{4} \right) + 2p\left( x- \tfrac{1}{2} \right) \right) \right\} \varphi(2x - 1)  \notag
		\\
		&\qquad+\tfrac{2}{\pi(n - 1)} \, \sum_{\substack{n \in \mathbb{Z}, n \neq 1}}  (-1)^{-n}  \cos\left( \tfrac{n\pi}{2} \right) 
		\exp\left\{- \tfrac{i}{2B} \left(\tfrac{A n^2}{4} + p n \right)\right\} \notag
		\\
		&\qquad \times 
		\exp\left\{ -\tfrac{i}{2B} \left( A \left( x^2 - \tfrac{n^2}{4} \right) + 2p\left( x - \tfrac{n}{2} \right) \right) \right\} \varphi(2x - n) \notag
		\\
		&=\exp\left\{ -\tfrac{i}{2B} \left( \tfrac{Ax^2}{4} + 2px \right) \right\} \text{sinc}(\pi(2x - 1)) \notag
		\\
		&\qquad + 2\sum_{\substack{n \in \mathbb{Z}, n \neq 1}} \exp\left\{-\tfrac{i}{2B} \left( Ax^2+2px \right) \right\} \text{sinc}\left( \pi(2x - n)\right).
	\end{align}
	Next, we illustrate the special affine wavelets defined in \eqref{V10} in Figure \ref{F1}, corresponding to two different configurations of the parametric matrix:
	$S = \left[\begin{smallmatrix} 1 & 1 & : & 2 \\ -1 & 0 & : & -1 \end{smallmatrix}\right]$ and
	$S' = \left[\begin{smallmatrix} 3 & 2 & : & 1 \\ 1 & 1 & : & -2 \end{smallmatrix}\right]$, respectively.
	
	\begin{figure}[htbp]
		\centering
		\includegraphics[width=0.75\textwidth]{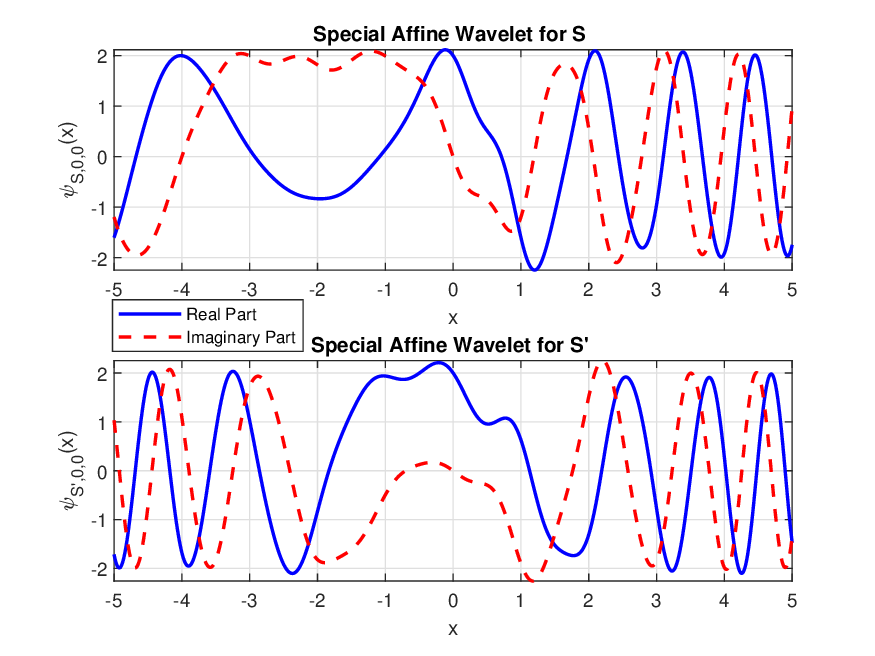}
		\caption{\small Special affine wavelets $\psi_{S,0,0}(x)$ corresponding to the scaling function $\varphi(x) = \text{sinc}(\pi x)$, for two SAFT parameter matrices $S$ and $S'$.}
		\label{F1}
	\end{figure}	
\end{example}

	\newpage	
	\begin{example}\label{S3EX6}
		Let $\varphi(x) = \chi_{[0,1)}(x)$, where $\chi_{[0,1)}(x)$ denotes the characteristic function of the interval $[0, 1)$. It is straightforward to verify that the system $\left\{ \varphi_{S,0,n}(x) : n \in \mathbb{Z} \right\}$ forms an orthonormal set. Consequently, it constitutes an orthonormal basis for the subspace $\mathscr{V}_{S}^0$ and thus serves as a scaling function associated with the MRA $\left\{ \mathscr{V}_{S}^n:n \in \mathbb{Z}\right\}$.
		
		\parindent=8mm\vspace{.1in}
		Since $\varphi(x) = \chi_{[0,1)}(x)$, the scaling coefficients $h_n$ become
		\begin{align*}
			h_n =
			\begin{cases}
				\tfrac{1}{\sqrt{2}}, & n = 0 \\[1.3ex]
				\tfrac{1}{\sqrt{2}} \, \exp\left\{-\tfrac{i}{2B}{ \left(\tfrac{A}{4}+ p \right)}\right\} & n = 1 \\[1.3ex]
				0, & \text{otherwise}
			\end{cases}.
		\end{align*}
		Consequently, the wavelet coefficients $d_n$ given by \eqref{S3E25} take the form
		\begin{align*}
			d_n=
			\begin{cases}
				\tfrac{1}{\sqrt{2}}, & n = 0 \\[1.3ex]
				\tfrac{1}{\sqrt{2}} \, \exp\left\{-\tfrac{i}{2B}{ \left(\tfrac{A}{4}+ p \right)}\right\}, & n = 1 \\[1.3ex]
				0, & \text{otherwise}
			\end{cases}.
		\end{align*}	
		Thus, the special affine wavelet corresponding to scaling function $\varphi(x) = \chi_{[0,1)}(x)$ is given by
		\begin{align}\label{V11}
			\psi_{S,0,0}(x) =
			\begin{cases}
				- \exp{\left\{-\tfrac{i}{2B} \left(Ax^2 + p\right)\right\}}, & 0 \leq x < \tfrac{1}{2} 
				\\[2mm]
				\exp\left\{{-\tfrac{i}{2B}\left(Ax^2 + p\right)}\right\}, & \tfrac{1}{2} \leq x < 1.
			\end{cases}
		\end{align}	
	It is evident that the special affine Haar wavelet provides enhanced adaptability compared to the classical Haar wavelet, owing to the freedom in selecting the parameters of the matrix $S$.
	
	\parindent=8mm\vspace{.1in}
	The graphical representation of the special affine wavelet defined in \eqref{V11}, corresponding to the parametric matrices
	$S = \left[\begin{smallmatrix} 1 & 1 & : & 2 \\ -1 & 0 & : & -1 \end{smallmatrix}\right]$ and
	$S' = \left[\begin{smallmatrix} 3 & 2 & : & 1 \\ 1 & 1 & : & -2 \end{smallmatrix}\right]$,
	is shown in Figure \ref{F2}.
	
	\begin{figure}[htbp]
		\centering
		\includegraphics[width=0.75\textwidth]{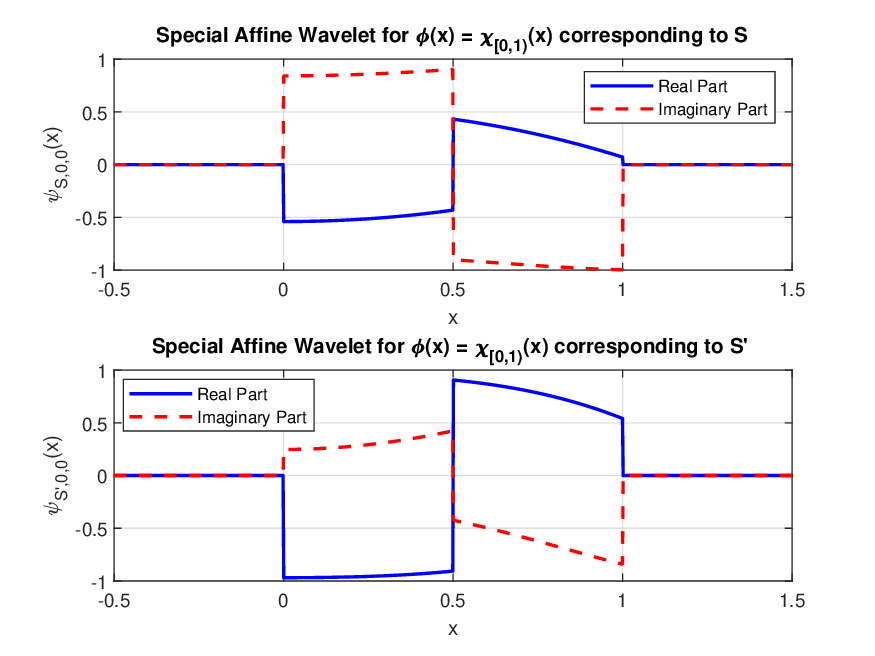}
		\caption{\small Special affine wavelets $\psi_{S,0,0}(x)$ corresponding to the scaling function $\varphi(x) = \chi_{[0,1)}(x)$, for two SAFT parameter matrices $S$ and $S'$.}
		\label{F2}
	\end{figure}	
	\end{example}
	\section{Function Approximation}
	To approx
	imate a function \( f(\T)\in L^2(\mathbb{R}) \) in $[0,1]$ by a series of special affine wavelet as it is done with Haar wavelet (see \cite{Lepik2014}), we divide the domain \({[0, 1]}\) into \(2M\) equal subintervals, each with a length of \(\Delta \T = {\frac{1}{2M}}\). The maximum level of resolution, \(\p\), is defined by \(M = 2^\p\). The special affine wavelet, denoted by \(h_{\Q}(\T)\), For \(\Q > 2\), it is defined as follows:

\begin{equation}\label{P1_eq1}
    h_{\Q}(\T) = 
     \begin{cases}
        -\Im(\psi_{S,0,0}(x)) & \text{if } \T \in  [\y_{1}(\Q), \y_{2}(\Q)),\\
        ~~\Im(\psi_{S,0,0}(x)) & \text{if }  \T \in [\y_{2}(\Q), \y_{3}(\Q)),\\
        ~~0 & \text{elsewhere}, 
    \end{cases}
\end{equation}
where,
\begin{align}\label{P1_eq2}
{ \y_{1}(\Q) =  2k\mu \Delta \mathbf{\T},}\;\; \;\;& { \y_{2}(\Q) =  (2k+1)\mu \Delta \mathbf{\T},}\nonumber\\ 
{\y_{3}(\Q) =  2 (k + 1)\mu \Delta \mathbf{\T},}\;\;\;\; & {\mu = M/m, \;\; \Delta \mathbf{\T} = 1/2M.}
\end{align}
Here $ \J = 0, 1,..., \p$ and $k = 0, 1,..., m-1$ stand for dilation and translations parameters, respectively, where \( m = 2^{\J} \). The index $ \Q $ is calculated as $ \Q = m+k+1 $. For $ \Q = 1 $ and $ \Q = 2 $ it is defined as follows
\begin{equation}\label{P1_eq3}
h_{1}(\T) = \chi_{[0,1)}(x),
\end{equation}
\begin{equation}\label{P1_eq4}
 h_{2}(\T) = 
     \begin{cases}
        -\Im(\psi_{S,0,0}(x))  & \text{for } \T \in  [0, 0.5),\\
        ~~\Im(\psi_{S,0,0}(x))  & \text{for }  \T \in [0.5, 1),\\
        ~~0 & \text{elsewhere }. 
    \end{cases}
\end{equation}
Now a square integrable function \( f(\T) \) over the interval {\([0, 1]\)} can be expanded in terms of special affine wavelets as follows,
\begin{equation}\label{P1_eq6}
f(\T) = \sum_{\Q=1}^{\infty} a_{\Q} h_{\Q}(\T).
\end{equation}
For computational purpose we truncate the infinite series
\begin{equation}\label{P1_eq6}
f(\T) = \sum_{\Q=1}^{2 M} a_{\Q} h_{\Q}(\T).
\end{equation}
and assume collocation points as \begin{equation}\label{collocation}
{\eta_{cl} = \frac{\tilde{\eta}_{cl-1} + \tilde{\eta}_{cl}}{2}, \quad \text{for } cl = 1, 2, \ldots, 2M.}
\end{equation}
where $\tilde{\eta}_{cl} = cl \Delta \T,~\text{for } cl = 0, 1, \ldots, 2M$. Now by using the collocation points in \eqref{P1_eq6} we obtain a system of linear equations where $a_{\Q}$ are unknowns. By solving the system of linear equations we get $a_{\Q}$ and by substituting the $a_{\Q}$ \eqref{P1_eq6} we get an approximation of the function. For $x^2$ we give some plots below:
\begin{center}
\begin{figure}[h]
\includegraphics[width=8cm]{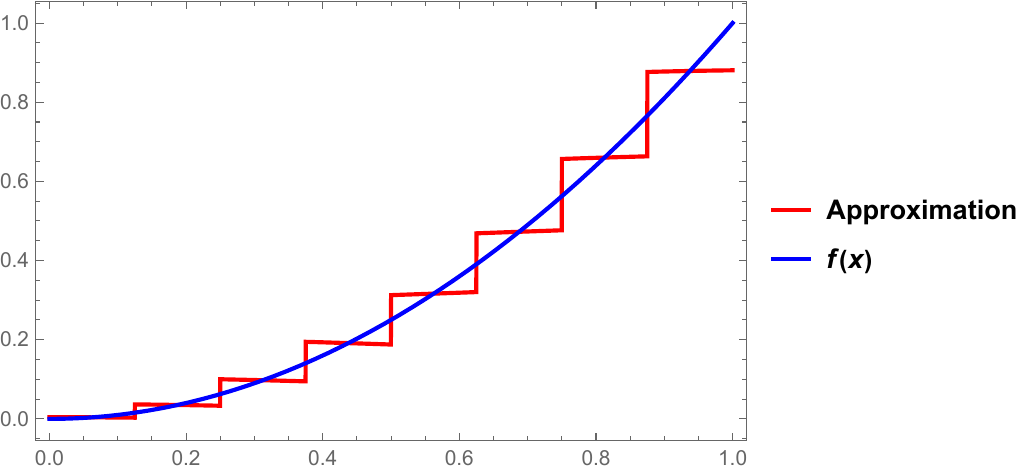}\includegraphics[width=8cm]{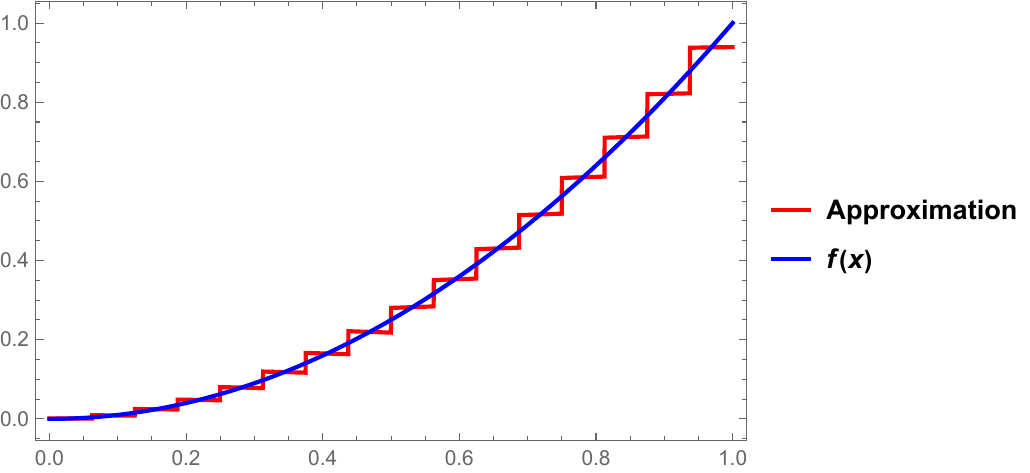}
\caption{{Affine wavelet approximation of $x^2$ on $[0,1]$ for $J=2$ and $J=3$ by using Imaginary part of $\psi$.}}
\end{figure}
 \end{center}
 
\begin{table}[h!]
\centering
 \begin{tabular}{|c|c|c|c|c|c|c|}
\hline
$L^\infty$ Error & J=1 & J=2 & J=3 & J=4 & J=5 & J=6\\
\hline
Special Affine & 0.132906 & 0.0675038 & 0.0303483 & 0.0163966 & 0.00740868 & 0.00412899\\
\hline
 Haar & 0.140625 & 0.0689062 & 0.0322266 & 0.0167871 & 0.00787354 &0.00422424\\
\hline
\end{tabular}
\caption{Comparison of $L^\infty$ error with Haar Wavelet for the function $x^2$ on $[0,1]$.}
\label{tab:sample}
\end{table}

	\section{Conclusion}\label{S4}
	In this work, we developed a comprehensive multiresolution framework within the SAFT domain, facilitating the construction of orthonormal bases in $L^2(\mathbb R)$. We established a sampling theorem for band-limited signals in the SAFT setting, which underpins the formulation of the SAMRA. Additionally, we proposed a method for generating orthogonal bases using this SAFT-based MRA structure. These theoretical developments were validated through illustrative examples, highlighting the practical applicability and flexibility of the proposed approach in time-frequency analysis.
	
	\parindent=0mm\vspace{.1in}
	
	{\small {\it Acknowledgments:} The second author is supported by Council of Scientific and
		Industrial Research (CSIR), Government of India, under file No. 09/1023(19437)/2024-EMR-I.
		
		\parindent=0mm\vspace{.1in}
		
		{\small {\it Competing interests:} The authors declare that they have no competing interests.

\end{document}